\documentclass[12pt,reqno]{amsart}
\usepackage{amsthm}
\usepackage{amssymb}
\usepackage[all]{xy}

\newtheorem{theorem}{Theorem}[section]
\newtheorem{proposition}[theorem]{Proposition}
\newtheorem{lemma}[theorem]{Lemma}

\newtheorem{definition}[theorem]{Definition}

 \newtheorem{remark}[theorem]{Remark}
\newtheorem{example}[theorem]{Example}
\numberwithin{equation}{section}

\begin{document}
 	
\baselineskip=15pt

\title[On the Seshadri constants of equivariant bundles]{On the Seshadri constants of
equivariant bundles over Bott-Samelson varieties and wonderful compactifications}

\author[I. Biswas]{Indranil Biswas}

\address{School of Mathematics, Tata Institute of Fundamental
Research, Homi Bhabha Road, Mumbai 400005, India}

\email{indranil@math.tifr.res.in}

\author[K. Hanumanthu]{Krishna Hanumanthu}

\address{Chennai Mathematical Institute, H1 SIPCOT IT Park, Siruseri, Kelambakkam 603103, 
India}

\email{krishna@cmi.ac.in}

\author[S. S. Kannan]{S. Senthamarai Kannan}

\address{Chennai Mathematical Institute, H1 SIPCOT IT Park, Siruseri, Kelambakkam 603103, 
India}

\email{kannan@cmi.ac.in}



\date{March 5, 2023}

\begin{abstract}
We study torus-equivariant vector bundles $E$ on a complex projective variety $X$ which is 
either a Bott-Samelson-Demazure-Hansen variety or a wonderful compactification of a 
complex symmetric variety of minimal rank. We show that $E$ is nef (respectively, ample) if 
and only if its restriction to every torus--invariant curve in $X$ is nef (respectively, 
ample). We also compute the Seshadri constants $\varepsilon(E,x)$, where $x\, \in\, X$ is 
any point fixed by the action of a maximal torus.
\end{abstract}

\maketitle
 	
\section{introduction}

Let $X$ be a complex projective variety. A vector bundle $E$ on $X$ is said to be nef (respectively, ample) if $\mathcal{O}_{\mathbb{P}(E)}(1)$ is 
a nef (respectively, ample) line bundle on the projective bundle $\mathbb{P}(E)$ over $X$. If $E$ is an ample vector bundle then it follows 
easily that the restriction of $E$ to any curve in $X$ is also ample. But the converse is not true in general, even when $E$ has rank one 
and $X$ has dimension two as shown by an example of Mumford (cf. \cite[Example 10.6]{Ha2} or \cite[Example 1.5.2]{La1}).

Nevertheless it is natural to ask the following question in specific situations: if the 
restriction of $E$ to every curve in $X$ is ample, then under what conditions is $E$ 
itself ample?

This question has been studied in several situations. An affirmative answer is given when 
$E$ is a torus-equivariant vector bundle either on a toric variety \cite{HMP} or on a 
generalized flag variety \cite{BHN}. An affirmative answer is also given when $E$ is any 
line bundle on a flag variety over a projective curve defined over the algebraic closure 
of a finite field \cite{BMP}. A related question was studied in \cite{BKN} for equivariant 
vector bundles on wonderful compactifications.

In this paper, we study the question and are able to give an affirmative answer in two different cases: 
Bott-Samelson-Demazure-Hansen (BSDH) varieties and wonderful compactification of symmetric varieties of minimal rank. 
Bott-Samelson-Demazure-Hansen varieties are natural desingularizations of Schubert varieties. They were first introduced by Bott and 
Samelson in differential geometric and topological context (see \cite{BS}). Demazure in \cite{Dem} and Hansen in \cite{Han} 
independently adapted the construction in algebro-geometric situation. The wonderful 
compactifications of symmetric spaces were constructed by De Concini and Procesi in \cite{DP}. For a semisimple adjoint group $G$, 
the wonderful compactification of $G=G\times G/\Delta(G)$ is an example of a complete symmetric space of minimal rank.

Seshadri constants for line bundles on projective varieties were introduced by Demailly in 
\cite{De90}. Let $X$ be a projective variety, and let $L$ be a nef line bundle on $X$. For 
a point $x\in X$, the {\it Seshadri constant} of $L$ at $x \in X$ is defined as: 
$$\varepsilon(X,L,x):= \inf\limits_{\substack{x \in C}} \frac{L\cdot C}{{\rm 
mult}_{x}C},$$ where the infimum is taken over all irreducible and reduced curves $C$ 
on $X$ passing through $x$.

Seshadri constants are useful in studying positivity questions and they are a focus of a 
lot of current research. In recent years, there has been significant work on computing 
Seshadri constants for line bundles on projective varieties, especially on 
surfaces. See \cite{B} for a detailed survey on these works.

While most of the works on Seshadri constants have considered the case of line
bundles, there has been recent interest in Seshadri constants for vector bundles of
arbitrary rank. An explicit definition of these was first given 
by Hacon \cite{Hac}; see Definition
\ref{sc-vb} below. 
For a comprehensive survey and a generalization to the relative setting, see \cite{FM}. 

In this paper, we compute the Seshadri constants of equivariant vector bundles at 
torus-fixed points in the two cases that we study.

In Section \ref{prelims}, we recall basic definitions, constructions and prove some 
preparatory lemmas. In Section \ref{main-results}, we prove our main results about 
positivity of equivariant vector bundles and their Seshadri constants. We show that a 
torus-equivariant vector bundle $E$ on a BSDH variety or on a wonderful compactification 
of a symmetric variety of minimal rank is ample if and only if its restriction $E|_C$ to every 
torus-invariant curve $C$ is ample (see Theorem \ref{thm-bsdh} and Theorem 
\ref{thm-wonderful}). We also compute the Seshadri constants of such vector bundles at 
torus-fixed points (Theorem \ref{sc-bsdh} and Theorem \ref{sc-wonderful}).

Each section is divided into two subsections dealing with the two cases that we study.

We work throughout over the field $\mathbb{C}$ of complex numbers. The field of real 
numbers is denoted by $\mathbb{R}$.

\section{preliminaries}\label{prelims}

In this section, we recall basic definitions and prove some preparatory results. 

\subsection{BSDH varieties}\label{prelims-bsdh}

Let $G$ be a semisimple and simply connected algebraic group defined over $\mathbb{C}$. We fix a maximal torus $T$ of $G$. The \textit{rank} of $G$ is defined to be the dimension of $T$.
 Let $W \,=\, N_G(T)/T$ be the Weyl group of $G$ with respect to $T$.
We denote by $R$ the set of roots of $G$ with respect to $T$ and by
$B$ a Borel subgroup of $G$ containing $T$, and $R^{+}\subset R$ the set of positive roots with respect to $B$. 
Let $$S \,=\, \{\alpha_1,\,\cdots,\,\alpha_n\}$$
be the set of all simple roots in $R^{+}$, where $n$ is the rank of $G$. 
The simple reflection in the Weyl group corresponding to a simple root $\alpha$ is denoted
by $s_{\alpha}$. For simplicity of notation, the simple reflection corresponding to a simple root $\alpha_{i}$ is denoted
by $s_{i}$. For a subset $J\subseteq S$, we denote the subgroup of $W$ generated by $s_i$, $\alpha_i\in J$ by $W_{J}$. For every subset $J\subseteq S$, there is a unique parabolic subgroup $P_J$ of $G$ containing $B$ whose Weyl group is $W_J$. Conversely, every parabolic subgroup of $G$ containing $B$ is of the form $P_J$ for some $J\subseteq S$; see \cite[8.4.3. Theorem, Page 147]{Spr}.

We denote the minimal parabolic subgroup $P_{\{\alpha_{i}\}}$ by $P_{\alpha_i}$.
For any subset $J\,\subseteq \,S$, let
$$W^{J}\,:=\,\{w\,\in\, W\,\, \mid\,\, w(\alpha)\,\in\, R^{+}\ \, \forall \ \, \alpha \,\in \, J\}.$$ For any $w
\,\in\, W$, set $R^{+}(w)\,:=\,\{\alpha \in R^{+}\,\,\big\vert\,\, w(\alpha)\,\in\, R^{-}\}$.

By $\mathfrak{g}$ (respectively, $\mathfrak{t}$)
we denote the Lie algebra of $G$ (respectively, $T$); similarly, the Lie algebra of $B$ is denoted
by $\mathfrak{b}$. Let $X(T)$ denote the group of all characters of $T$. 

For a root $\alpha,$ we denote by 
$U_{\alpha}$ (respectively, $\mathfrak{g}_{\alpha}$) the one-dimensional 
$T$-invariant root subgroup of $G$ (respectively, 
subspace of $\mathfrak{g}$) on which $T$ acts through the character $\alpha.$ 

For any $w \,\in\, W$, let $$X(w)\,:=\,\overline{BwB/B}\, \subset\, G/B$$ be the corresponding
Schubert variety in $G/B$. Given a reduced expression 
$w\,=\, s_{{i_1}}s_{{i_2}}\cdots s_{{i_r}}$ of $w$, the Bott-Samelson-Demazure-Hansen variety
\begin{equation}\label{ebs}
Z(w,\,{\underline i})
\end{equation}
is a desingularization of the Schubert variety $X(w)$,
where $\underline i\,:=\,(i_1,\, \cdots,\, i_r)$.

Then $Z(w,\,{\underline i})$ in \eqref{ebs}
is defined as $$Z(w,\,\underline{i})\,=\, \frac{P_{\alpha_{i_1}}\times 
P_{\alpha_{i_2}}\times\cdots \times P_{\alpha_{i_{r}}}}{B\times B\times \cdots \times B},$$
where the action of $\underbrace{B\times B\times \cdots \times B}_{r\text{ copies }}$ on $P_{\alpha_{i_{1}}}\times P_{\alpha_{i_{2}}}\times\cdots\times P_{\alpha_{i_{r}}}$ 
is given by
$$(p_{1},\, p_{2},\, \cdots ,\, p_{r})(b_{1},\, b_{2},\, \cdots,\, b_{r})
\,=\,(p_{1}\cdot b_{1},\, b_{1}^{-1} \cdot p_{2}\cdot b_{2},\, \cdots,\, b_{r-1}^{-1}\cdot p_{r}\cdot b_{r})$$
for $p_{j}\,\in \, P_{\alpha_{i_{j}}}$, $b_{j}\,\in\, B$ and $\underline{i}\,=\,(i_{1},\,i_{2},\,\cdots,\,i_{r})$
(see \cite[Definition 1, Page 73]{Dem}, \cite[Definition 2.2.1, Page 64]{Bri}).
We recall that $Z(w,\,\underline{i})$ is a smooth projective variety.

Let $\pi\,:\,P_{\alpha_{i_{1}}}\times P_{\alpha_{i_{2}}}\times\cdots\times P_{\alpha_{i_{r}}}\,\longrightarrow\, Z(w,\, \underline{i})$ be the 
natural morphism. Let 
\begin{equation}\label{ebs2}
\phi_{w}\, \, :\,\, Z(w,\, \underline{i})\, \longrightarrow\, X(w)
\end{equation}
be the natural birational surjective morphism; see \cite[Page 67]{Bri} for more details. 

For the sake of simplicity, we will refer to any Bott-Samelson-Demazure-Hansen 
variety simply as a BSDH variety. For more details on BSDH varieties, see \cite{BS,Dem,Han}. 

Let
\begin{equation}\label{ebs3}
f_{r}\,:\,Z(w,\, \underline{i})\,\longrightarrow \,Z(ws_{i_{r}},\, \underline{i'})
\end{equation}
be the morphism induced by the projection
$$P_{\alpha_{i_{1}}}\times P_{\alpha_{i_{2}}}\times\cdots\times P_{\alpha_{i_{r}}}
\,\longrightarrow\, P_{\alpha_{i_{1}}}\times P_{\alpha_{i_{2}}}\times \cdots \times P_{\alpha_{i_{r-1}}},$$
where $\underline{i'}
\,=\,(i_{1},\,i_{2},\,\cdots,\, i_{r-1}).$ The map $f_{r}$ in \eqref{ebs3} is in fact a $P_{\alpha_{i_{r}}}/B\,\simeq\, \mathbb{P}^1$
fibration. There is a natural $B$--equivariant imbedding 
\begin{equation}\label{be}
Z(w,\, \underline{i})\,\hookrightarrow \,Y \,=\, Y(w,\,\underline{i})\,:=\, G/P_{S\setminus\{\alpha_{i_1}\}}\times 
G/P_{S\setminus\{\alpha_{i_2}\}}\times \cdots\times G/P_{S\setminus\{\alpha_{i_r}\}}\, ;
\end{equation}
see \cite[Theorem 1(ii), Page 608]{M}.

For any $\beta\,\in\, R^{+}$, let $C_{\beta}\,:=\,\overline{U_{\beta}s_{\beta}B/B}\,\subset\, G/B$.
For any $v\,\in\, W^{S\setminus\{\alpha_{i_j}\}}$ and $\beta\,\in\, R^{+}(v^{-1})$, let $$C_{\beta, v, j}\,:=\,
\overline{U_{\beta}vP_{S\setminus\{\alpha_{i_j}\}}/P_{S\setminus\{\alpha_{i_j}\}}}\,\subset\, 
G/P_{S\setminus\{\alpha_{i_j}\}}\, .$$

Note that $C_{\beta}$ is isomorphic to the curve $C_{z,\beta}$ in the notation of Lemma \ref{lem-wonderful} for the special case of wonderful compactification of a semisimple adjoint group. So $C_{\beta}$ is 
isomorphic to the projective line $\mathbb{P}^1$ for every $\beta$ by Lemma \ref{lem-wonderful}. 

Note that by Bruhat decomposition (see \cite[8.3.8 Theorem, Page 145]{Spr}), every $T$--invariant curve in 
$G/P_{S\setminus \{\alpha_{i_j}\}}$ is of the form $C_{\beta, v, j}$ for some $v\in W^{S\setminus \{\alpha_{i_j}\}}\setminus \{1\}$ and $\beta\in R^{+}(v^{-1})$.

Let $v\,\in\, W^{S\setminus \{\alpha_{i_j}\}}$. For any $\beta\,\in\, R^{+}(v^{-1})$, 
there is an isomorphism $$\tau_{\beta, v, j}\,:\,U_{\beta}s_{\beta}B/B\,\longrightarrow\,
U_{\beta}vP_{S\setminus \{\alpha_{i_{j}}\}}/ P_{S\setminus \{\alpha_{i_{j}}\}}$$ such that $\tau_{\beta, v, j}(gs_{\beta}B/B)\,=\, 
gvP_{S\setminus \{\alpha_{i_{j}}\}}/P_{S\setminus \{\alpha_{i_{j}}\}}$ for all $g\,\in\, U_{\beta}$. Now, $\tau_{\beta, v, j}$
can be extended uniquely to an isomorphism $$\sigma_{\beta, v, j}\,:\,C_{\beta}\,\longrightarrow\, C_{\beta, v, j},$$ by the valuative criterion.

Let $\pi_{j}\,:\,Y\,\longrightarrow\, G/P_{S\setminus\{\alpha_{i_j}\}}$ be the projection to the $j$--th factor
in \eqref{be}.

\begin{lemma}Let $\beta\in R^{+}$.
Let $v:= (v_1,\, v_2,\, \cdots ,\, v_r)\,\in\, W^{S\setminus \{\alpha_{i_1}\}}\times W^{S\setminus \{\alpha_{i_1}\}}
\times \cdots\times W^{S\setminus \{\alpha_{i_r}\}}$ be such that there is an integer $1\,\leq\, j \,\leq\, r$ for which 
$v_j^{-1}(\beta)$ is negative. Let $A\,\subset \,\{1,\, 2,\, \cdots ,\, r\}$ be a non-empty
subset such that $\beta\,\in\, \bigcap_{j\in A}R^{+}(v_j^{-1})$. Then there is a unique $T$--equivariant morphism
$$\phi_{\beta, v, A}\,\,:\,\,C_{\beta}\,\longrightarrow\, Y$$ (see \eqref{be}) satisfying the following conditions:
\begin{itemize}
\item $\pi_j\circ \phi_{\beta, v,A}\,=\,\sigma_{\beta, v_j, j}$ for every $j\,\in\, A$, and

\item $\pi_j\circ \phi_{\beta, v,A}\,=\,v_{j}P_{S\setminus \{\alpha_{i_{j}}\}}/P_{S\setminus \{\alpha_{i_{j}}\}}$
for every $j\,\in\, \{1,\, 2,\, \cdots,\, r\}\setminus A$.
\end{itemize}
\end{lemma}

\begin{proof}
For $x\,\in\, C_{\beta}$, define $\phi_{\beta, v,A}(x)\,=\,(x_1,\, x_2, \,\cdots,\, x_r)$, where $x_j\,=\,\sigma_{\beta, v_j, j}(x)$ if
$j\,\in\, A$, and $x_j\,=\,v_{j}P_{S\setminus \{\alpha_{i_{j}}\}}/P_{S\setminus \{\alpha_{i_{j}}\}}$ if $j\,\in
\,\{1,\, 2,\, \cdots,\, r\}\setminus A$.
\end{proof} 

\begin{lemma}\label{lem-bsdh}
For any $T$--invariant curve $C$ in $Y$ (see \eqref{be}), there exist the following:
\begin{enumerate}
\item $v_j\,\in\, W^{S\setminus \{\alpha_{i_j}\}}$ and $t_{j}\,\in\, T$ for every $1\,\leq\, j \,\leq\, r$,

\item a non-empty subset $A\,\subset\, \{1,\, 2,\, \cdots,\, r\}$, and

\item a root $\beta\,\in\, \bigcap_{j\in A}R^{+}(v_{j}^{-1})$,
\end{enumerate} 
such that 
$$C\,=\,(t_1,\, t_2,\, \cdots,\, t_r).\phi_{\beta,v,A}(C_{\beta}),$$
where $v\,=\,(v_1,\, v_2,\, \cdots,\, v_r).$
\end{lemma}

\begin{proof}
Let $z_k\,=\,P_{S\setminus \{\alpha_{i_{k}}\}}/P_{S\setminus \{\alpha_{i_{k}}\}}$ for every $1\,\leq\, k \,\leq\, r$. Let 
$$A\,:=\,\{j\,\in\, \{1,\, 2,\, \cdots ,\, r\}\,\,\mid\,\, \pi_{j}(C)\ \, \text{is a curve}\}.$$
So, for any $j\,\in\, \{1,\, 2, \,\cdots,\, r\}\setminus A$, the image $\pi_{j}(C)$ is a $T$--fixed point, and hence
it is of the form $v_jz_j$ for some $v_j\,\in\, W^{S\setminus \{\alpha_{i_j}\}}$. Further, for every $j\,\in\, A$, there is an element
$v_j\,\in\, W^{S\setminus \{\alpha_{i_j}\}}$ and a root $\beta_j\,\in\, R^{+}(v_j^{-1})$ such that $\pi_j(C)\,=\, 
\sigma_{\beta_j, v_j, j}(C_{\beta_j})$. 

Note that $\pi_j\,:\,Y(w,\, \underline{i})\,\longrightarrow\, G/P_{S\setminus \{\alpha_{i_j}\}}$ is $T$--equivariant 
for every $1\,\leq\, j \,\leq \,r$. Further, note that $\pi_{j}^{-1}(U_{\beta_j}v_jP_{S\setminus \{\alpha_{i_j}\}}/P_{S\setminus \{\alpha_{i_j}\}})\cap C$ is a $T$--invariant open subset of $C$, which is isomorphic to $\mathbb{A}^{1}$. So, the action of $T$ on $\pi_{j}^{-1}(U_{\beta_j}v_jP_{S\setminus \{\alpha_{i_j}\}}/P_{S\setminus \{\alpha_{i_j}\}})\cap C$ is given by the character $\beta_j$ for every $j\in A$.
In particular, we have $\beta_j\,=\,\beta_k$ for every pair of distinct elements
$j$ and $k$ of $A$.

Let $\beta\,:=\,\beta_j$ for some $j\,\in\, A$. Let $u_{\beta}\,:\,\mathbb{A}^{1}\,\longrightarrow\, U_{\beta}$ be an isomorphism, and take $x_0\,=\,u_{\beta}(1)s_{\beta}B/B$. Now, let $y_0
\,\in\, C$ be a point such that $\pi_{j}(y_0)$ is not fixed by $T$ for any $j \,\in\, A$. Then for any $j \,\in\, A$, there is an element
$t_j\,\in\, T$ for which $\pi_{j}(y_0)\,=\,t_j\sigma_{\beta, v_j,j}(x_0)$. Take $t_j\,=\,1$ for all $j\,\in
\,\{1,\, 2,\, \cdots, \,r\}\setminus A$. 
Let $v\,=\,(v_1,\, v_2,\, \cdots,\, v_r)$. Therefore, $(t_1,\, t_2,\, \cdots,\, t_r).\phi_{\beta, v, A}(x_0)\,\in\, C$ and it is not fixed
by $T$. The condition that $C$ is $T$--invariant implies that the $T$--orbit of $T\cdot(t_1,\, t_2,\, \cdots, \,t_r).\phi_{\beta, v, A}(x_0)$
is an open dense
subset of $C$. Since $(t_1,\, t_2, \,\cdots,\, t_r)$ commutes with every element of $T$ (this is because $T$ is the diagonal subgroup
of $T^{r}$), we conclude that
$$T\cdot(t_1,\, t_2,\, \cdots,\, t_r).\phi_{\beta, v, A}(x_0)\, =\,
(t_1,\, t_2,\, \ldots, \,t_r).T\cdot\phi_{\beta, v, A}(x_0)$$ is an open dense subset of 
$(t_1, \,t_2,\, \cdots,\, t_r). \phi_{\beta, v, A}(C_{\beta}) $. Thus, we have $$C\,=\,(t_1,\, t_2,
\,\cdots, \,t_r).\phi_{\beta,v,A}(C_{\beta}),$$ and the proof is complete.
\end{proof}

\begin{remark} \rm
The set of all $T$--invariant curves in $Z(w,\, \underline{i})$ is not necessarily finite.
\end{remark}

\begin{example}\rm $G=SL(3,\mathbb{C})$, $w=s_1s_2s_1$ and $\underline{i}=(1,2,1)$.
Let $Z:=\pi(P_{\alpha_1}\times \{1\} \times P_{\alpha_1})$. 
Then $Z$ is a $T$--invariant closed subset of 
$Z(w,\,\underline{i})$ which is isomorphic to $\mathbb{P}^{1}\times \mathbb{P}^{1}$.

Note that the action of $T$ on $\mathbb{P}^{1}\times \mathbb{P}^{1}$ is the following. 

Let $$t \,=\, \begin{bmatrix}
a & 0&0\\
0 & b&0\\
0&0&c
\end{bmatrix} \,\in\, T$$ and $p \,=\, ([x_0:x_1],\,[y_0:y_1]) \,\in\, \mathbb{P}^{1}\times \mathbb{P}^{1}$. Then 
$$t\cdot p \,=\, ([ax_0: bx_1],\,[ay_0: by_1]).$$ 

Consider the subgroup $$T'\,: =\, \left\{
\begin{bmatrix}
a & 0&0\\
0 & b&0\\
0&0&1
\end{bmatrix} \,\Big\vert\,\, a,\, b \in \mathbb{C}^{\times}\, \text{ and }\, ab\,=\,1\right\}$$ of $T$. 
For every $t \,\in \,T'$, $$C_t\,:= \,\{(p, t\cdot p) \,\,\mid\,\, p \,\in\, \mathbb{P}^1\}$$ is a $T$--invariant 
curve in $\mathbb{P}^{1}\times \mathbb{P}^{1}$. 
If $C_{t_1} = C_{t_2}$ for some $t_1, t_2 \in T'$, then
either $t_1 = t_2$ or $t_1 = \begin{bmatrix}
-1 & 0&0\\
0 & -1&0\\
0&0&1
\end{bmatrix} t_2$. 

Therefore there are infinitely many distinct 
$T$--invariant curves in $\mathbb{P}^{1}\times \mathbb{P}^{1}$. 
\end{example}

\subsection{Wonderful compactifications}\label{prelims-wonderful}

Let $G$ be a semisimple adjoint type algebraic group over $\mathbb{C}$. 
We use the notation introduced at the beginning of Section \ref{prelims-bsdh}. 

Let $\sigma$ be an automorphism of $G$ of order two. Let $H\, =\, G^\sigma\, \subset\, G$ be the subgroup
defined by the fixed point locus of $\sigma$. We now recall some properties of the symmetric variety $G/H$.

A torus $T^{\prime}$ of $G$ is said to be \textit{$\sigma$--anisotropic} if 
$\sigma(t) \,=\, t^{-1}$ for every $t\,\in\, T^{\prime}.$
The rank of $G/H$ is the dimension of a maximal dimensional $\sigma$--anisotropic torus.

If $T$ is a $\sigma$---invariant maximal torus of $G,$ then $\sigma$ induces an automorphism of 
$X(T)$ of order two. Evidently, $\sigma(R)\,=\,R.$ Further, we have
\begin{equation}\label{t12}
T\,=\,T_{1}T_{2},
\end{equation}
where $T_{1}$ and $T_2$ are tori with the
property that $\sigma(t)\,=\,t$ for every 
$t\,\in\, T_{1}$ and $\sigma(t)\,=\,t^{-1}$ for every 
$t\,\in\, T_{2}.$ Clearly, $T_{1}\bigcap T_{2}$ is a finite group, and $\text{rank}(G/H)\,\ge\, \text{rank}(G)-\text{rank}(H)$. Recall that the rank of $G$ is the dimension of a maximal torus of $G$. 

Throughout, we assume that $G/H$ is of \textit{minimal rank}, meaning 
$$\text{rank}(G/H)\,=\,\text{rank}(G)-\text{rank}(H).$$ 
See \cite{Br,Kn,Ri,RS} for more details.

Let
\begin{equation}\label{x}
X\,:=\,\overline{G/H}
\end{equation}
be the wonderful compactification of $G/H$, constructed in \cite{DP}. 
Let $T$ be a $\sigma$--invariant maximal torus of $G$ for which
$\dim T_1$ (see \eqref{t12}) is maximal. 
Recall that $W$ is the Weyl group of $G$. 
Let $R$ be the root system of $G$ relative to $T$. 
For a Borel subgroup $B$ of $G$ containing $T$, let $R^{+}(B)\subset R$ be the set of positive roots corresponding to $B$.
Choose a Borel subgroup $B$ of $G$ containing
$T$ satisfying the condition that for any root $\alpha\in R^{+}(B)$
either $\sigma(\alpha)\,=\,\alpha$ or $\sigma(\alpha)\,\in\, -R^{+}(B).$ 
Recall that a parabolic subgroup $P$ of $G$ containing $B$ is said to be $\sigma$--split parabolic subgroup if $\sigma(P)$ is opposite to $P$, i.e., $\alpha$ is a root of $P$ if and only if $-\alpha$ is a root of $\sigma(P)$. Let $P$ be a minimal $\sigma$--split parabolic subgroup of $G$ containing $B$. Let $L:=P\cap \sigma(P)$ be the 
$\sigma$--invariant Levi subgroup of $G$ containing $T$. Let $B_{L}:=B\cap L$. Note that $B_{L}$ contains $T$. Let $R^{+}(B_{L})\subset R^{+}(B)$
be the set of positive roots corresponding to the Borel subgroup $B_{L}$ of $L$.
We will use the following result which is a consequence of \cite[Lemma 2.1.1, Page 482]{BJ}.

\begin{lemma}\label{lem-wonderful}
Let $z$ be the unique $B$--fixed point of $X$ in \eqref{x}. For a positive root $\alpha \,\in\, R^{+}(B)\setminus R^{+}(B_{L})$, let $C_{z, 
\alpha}\,=\,\overline{U_{\alpha}s_{\alpha}z}.$ For $\gamma\,=\,\alpha-\sigma(\alpha)$, let $C_{z, \gamma}$ be the unique 
$T$--invariant curve containing $z$ on which $T$ acts through the character $\gamma.$
Then the irreducible $T$--invariant curves in $X$ are the $W$--translates of the curves $C_{z,\alpha}$ and $C_{z,\gamma}$. They are all 
isomorphic to $\mathbb{P}^1$.
\end{lemma}

\section{Equivariant vector bundles}\label{main-results}

\subsection{Equivariant bundles on BSDH varieties}\label{bsdh}

In this subsection we investigate positivity of torus--equivariant vector bundles on BSDH varieties. 
We follow the notation of Section \ref{prelims-bsdh}.

\begin{theorem}\label{thm-bsdh}
Let $G$ be a semisimple and simply connected algebraic group over the complex numbers. Let $T$ be a maximal torus of $G$,
and let $B$ be a Borel subgroup of $G$ containing $T$ as in Section \ref{prelims-bsdh}. Let $w\,=\,s_{{i_1}}s_{{i_2}}\cdots s_{{i_r}}$ be
a reduced expression and $\underline i\,:=\,(i_1,\,\cdots,\,i_r)$. Let
$Z(w,\,{\underline i})$ be the corresponding Bott-Samelson-Demazure-Hansen variety. 

Let $E$ be a $T$--equivariant vector bundle on $Z(w,\,{\underline i})$. Then $E$ is nef
(respectively, ample) if and only if the restriction $E\big\vert_C$ of $E$
to every $T$--invariant curve $C$ on $Z(w,\,{\underline i})$ is nef (respectively, ample). 
\end{theorem}

\begin{proof}
Let $Y(T)$ be the group of all $1$--parameter subgroups of $T$. Note that
$Y(T)$ is a finitely generated free abelian group whose rank coincides with the
dimension of $T$. Let
\begin{equation}\label{b}
\{\lambda_1,\,\cdots,\,\lambda_n\}
\end{equation}
be a basis of the $\mathbb Z$--module $Y(T)$. 

If $E$ is nef (respectively, ample), then clearly $E\big\vert_C$ is nef (respectively, ample) for every
curve $C$ on $Z(w,\,{\underline i})$.

To prove the converse, first assume that $E$ is a $T$--equivariant vector
bundle on $Z(w,\,{\underline i})$ such that its restriction $E\big\vert_C$ to every
$T$--invariant curve $C\,\subset\, Z(w,\,{\underline i})$ is nef.

Let $$p\,:\, \mathbb{P}(E) \,\longrightarrow\, Z(w,\,{\underline i})$$ be the projective bundle
over $Z(w,\,{\underline i})$ 
associated to $E$. 
Let $\mathcal{O}_{\mathbb{P}(E)}(1)$ denote the tautological
relative ample line bundle over $\mathbb{P}(E)$. By
definition, $E$ is nef if the line bundle $\mathcal{O}_{\mathbb{P}(E)}(1)$ 
on $\mathbb{P}(E)$ is nef. 
So to prove that $E$ is nef, it suffices to
show that $\mathcal{O}_{\mathbb{P}(E)}(1)\big\vert_D$ is nef for every curve
$D\,\subset\, \mathbb{P}(E)$. If $p(D)$ is a point, then 
$\mathcal{O}_{\mathbb{P}(E)}(1)\big\vert_D$ is ample, because 
$\mathcal{O}_{\mathbb{P}(E)}(1)$ is relatively ample. 

Suppose now that $p(D)$ is 
a curve in $Z(w,\,{\underline i})$. Let $\widetilde{D_1}$ be
the flat limit of the curves $\lambda_1(t)D$ (see \eqref{b}) as $t \,\to\, 0$. In other
words, $\widetilde{D_1}$ is a $1$--cycle which corresponds to the limit of the points
$\lambda_1(t)[D]$, as $t \,\to\, 0$, in the Hilbert scheme of curves in $\mathbb{P}(E)$. 
Note that since $E$ is $T$--equivariant, the $1$--parameter subgroup
$\lambda_1$ acts on the Hilbert scheme of curves in
$\mathbb{P}(E)$. It follows that the $1$--cycle $\widetilde{D_1}$ 
and $D_1: \,= \, p(\widetilde{D_1})$ are in fact $\lambda_1$--invariant. Now let
$\widetilde{D_2}$ be the flat limit of $\lambda_2(t)\widetilde{D_1}$
as $t \,\to\, 0$. Since $\lambda_1$ and $\lambda_2$ commute, we see that 
$\widetilde{D_2}$ and $p(\widetilde{D_2})$ are invariant under both
$\lambda_1$ and $\lambda_2$. Continuing this way, we obtain a $1$--cycle
$\widetilde{D_n}$ on $\mathbb{P}(E)$ such that both 
$\widetilde{D_n}$ and $p(\widetilde{D_n})$ are invariant under
$\lambda_1,\,\cdots, \,\lambda_n$. 
Since $\{\lambda_1,\,\cdots, \,\lambda_n\}$ is a basis of $Y(T)$, 
both $\widetilde{D_n}$ and $p(\widetilde{D_n})$ are in fact $T$--invariant. 

Now, by the assumption on $E$, we have
\begin{equation}\label{in}
{\rm degree}(\mathcal{O}_{\mathbb{P}(E)}(1)\big\vert_{\widetilde{D_n}}) \,\ge\, 0\, .
\end{equation}
Since the curve $D$ is linearly equivalent to $\widetilde{D_n}$, from \eqref{in} it follows immediately
that ${\rm degree}(\mathcal{O}_{\mathbb{P}(E)}(1)\big\vert_D)\,\ge\, 0$. This proves that $E$ is nef if $E\big\vert_C$ is 
nef for every $T$--invariant curve $C \,\subset\, Z(w,\,{\underline i})$.

Next suppose that $E\big\vert_C$ is ample for every $T$--invariant curve $C \,\subset\, Z(w,\,{\underline i})$. We
claim that $E$ is ample. 

Recall from \eqref{be} that we have an embedding
$$Z(w,\, \underline{i})\,\hookrightarrow\, Y(w,\,\underline{i}) \,=\, G/P_{S\setminus\{\alpha_{i_1}\}}\times 
G/P_{S\setminus\{\alpha_{i_2}\}}\times \cdots\times G/P_{S\setminus\{\alpha_{i_r}\}}.$$
It can be shown inductively that the operation of restriction 
of line bundles gives an isomorphism of the Picard groups
\begin{equation}\label{er}
\text{res}\,\,:\,\, \text{Pic}(Y(w,\,\underline{i})) \,\longrightarrow\, \text{Pic}(Z(w,\, \underline{i}));
\end{equation}
see \cite[Section 3.1, Page 464]{LT} for more details. 

Further, every line bundle on $Y(w,\,\underline{i})$ is $T^r$--equivariant. 
In particular, every line bundle on $Z(w,\,{\underline i})$ is
$T$--equivariant, for the action defined using the diagonal inclusion of $T$ in $T^r$. Hence if $F$ is
a $T$--equivariant vector bundle on $Z(w,\,{\underline i})$, then so is
$F \otimes L$ for any line bundle $L$ on $Z(w,\,{\underline i})$. 

To prove that $E$ is ample, fix an ample line bundle $L$ on $Z(w,\,\underline{i})$. 
Let $\text{Sym}^n(E)$ denote the $n$--th symmetric power of $E$. 

In view of the isomorphism res in \eqref{er}, we may assume that $L$ is actually a line bundle on $Y(w,\,\underline{i})$. 
For any $T$--invariant curve $C$ in $Y(w,\,\underline{i})$, let $a_C$ denote the
positive integer such that $L\big\vert_C \,=\, \mathcal{O}_C(a_C)$; recall that $C\,\cong\, {\mathbb P}^1$ and $L\big\vert_C$
is ample. Then for $t \,\in\, T^r$, we have $a_C \,=\, a_{t\cdot C}$ for every $T$--invariant curve $C$ in $Y(w,\,\underline{i})$.

By Lemma \ref{lem-bsdh}, up to $T^r$--translates, 
there are only finitely many $T$--invariant curves in $Y(w,\,\underline{i})$. Therefore the collection 
of integers $a_C$ as $C$ varies over $T$--invariant curves in $Y(w,\,\underline{i})$ is finite. 
In particular, this collection of integers $a_C$, as $C$ varies over $T$--invariant curves in $Z(w,\,\underline{i})$, is also finite.

Take a $T$--invariant curve $C\, \subset\, Z(w,\,\underline{i})$.
Since $E\big\vert_C$ is ample, the vector bundle $\text{Sym}^n(E)\big\vert_C$ is also ample for every
integer $n\, \geq\, 1$. Moreover, $\text{Sym}^n(E)\big\vert_C$ is a direct sum of line bundles and the degrees of these line bundles 
increase with $n$. Since the collection of integers $a_C$ such that $L\big\vert_C \,=\, \mathcal{O}_C(a_C)$ is finite as $C$ varies over 
the $T$--invariant curves in $Z(w,\,\underline{i})$, we can choose a sufficiently large integer $n$ such that 
$\text{Sym}^n(E)\otimes L^{-1}\big\vert_C$ is nef for every $T$--invariant curve $C$. By
the first part of the theorem, this vector bundle $\text{Sym}^n(E)\otimes L^{-1}$ is nef. Since
$L$ is ample, this implies that $\text{Sym}^n(E)\,=\, (\text{Sym}^n(E)\otimes L^{-1})\otimes L$
is ample, and consequently $E$ itself is ample (see \cite[Proposition 6.2.11]{La} and 
\cite[Proposition 2.4, Page 67]{Ha}). 
\end{proof}

\begin{remark}\rm
A similar result was proved for equivariant vector bundles on toric varieties in \cite[Theorem 2.1]{HMP} and on generalized flag varieties in 
\cite[Theorem 3.1]{BHN}. Our proof of Theorem \ref{thm-bsdh} above is motivated by these results.
\end{remark}

\begin{proposition}\label{blowup-bsdh}
Let $G$, $T$, $B$, $Z(w,\,\underline{i})$ and $Y(w,\,\underline{i})$ be as in Theorem \ref{thm-bsdh}.
Let $x \,\in\, Z(w,\,\underline{i})$ be a $T$--fixed point, and let $\pi\,:\, \widetilde{X}\,\longrightarrow\,
Z(w,\,\underline{i})$ denote the blow-up of $Z(w,\,\underline{i})$ at $x$. Then the following three statements hold:
\begin{enumerate}
\item The action of $T$ on $Z(w,\,\underline{i})$ lifts to $\widetilde{X}$. 

\item Let $F$ be a $T$--equivariant vector bundle on $\widetilde{X}$. Then $F$ is nef if and 
only if the restriction $F\big\vert_{\widetilde{C}}$ of $F$ to every $T$--invariant curve 
$\widetilde{C}\,\subset\, \widetilde{X}$ is nef.

\item Let $W_x$ denote the exceptional divisor of the above blow-up $\pi$. Let $E$ be a 
$T$--equivariant vector bundle on $Z(w,\,\underline{i})$. Then $(\pi^{\star}E)\otimes 
\mathcal{O}_{\widetilde{X}}(m\cdot W_x)$ is a $T$--equivariant vector bundle on $\widetilde{X}$ for 
every integer $m$.
\end{enumerate}
\end{proposition}

\begin{proof}
Let $x \,\in \,Z(w,\,\underline{i}) $ be a $T$--fixed point. Since $\pi$ is a blow up centered at $x$, 
the action of $T$ on $Z(w,\,\underline{i})$ lifts to $\widetilde{X}$. This proves (1). 

Since the action of $T$ lifts to $\widetilde{X}$, the proof of (2) goes through
exactly as the proof of the analogous statement in Theorem
\ref{thm-bsdh} does. Note that in the proof of the nef part of Theorem \ref{thm-bsdh} we only used
the $T$--action on $Z(w,\,\underline{i})$. 

Now we will prove (3). 
Since $T$ acts on the exceptional divisor $W_x$ of $\pi$, 
we conclude that $\mathcal{O}_{\widetilde{X}}(m\cdot W_x)$ is a $T$--equivariant line
bundle for every integer $m$. Hence if $E$ is a $T$--equivariant
vector bundle on $Z(w,\,\underline{i})$, then $(\pi^{\star}E) \otimes \mathcal{O}_{\widetilde{X}}(m\cdot W_x)$ is a 
$T$--equivariant vector bundle on $\widetilde{X}$ for every integer $m$.
\end{proof}

We now recall the definition of Seshadri constant for vector bundles. 
Let $E$ be a vector bundle on a projective variety $X$. Take any point $x \,\in\,
X$. The Seshadri constant of $E$ at $x$ was first defined explicitly in
\cite{Hac}. 

Let $\pi\,:\, \widetilde{X} \,\longrightarrow\, X$ denote the blow-up of $X$ at $x$.
Consider the following diagram:
\begin{equation}\label{fig1}
\xymatrix{
\mathbb{P}(\pi^{\star}E) \ar[r]^{\widetilde{\pi}}\ar[d]_{q} & \mathbb{P}(E)\ar[d]^{p}\\
\widetilde{X} \ar[r]^{\pi} & X
 }
\end{equation}
so the above projections $p$ and $q$ define projective bundles. Let
\begin{equation}\label{xi}
\xi \,=\, \mathcal{O}_{\mathbb{P}(\pi^{\star}E)}(1)
\end{equation}
be the tautological bundle on $\mathbb{P}(\pi^{\star}E)$. Let 
$Y_x \,=\, p^{-1}(x)$ and $Z_x \,=\,\widetilde{\pi}^{-1}(Y_x)$.

We can now give the definition of Seshadri constants of vector bundles. 

\begin{definition}\label{sc-vb}
Let $X$ be a projective variety and let $E$ be a vector bundle on $X$. For a point $x \in X$, 
the 
Seshadri constant of $E$ at $x$ is defined as
$$\varepsilon(E,x) \,\,:=\,\, \textup{sup}\{\lambda \,\ge\, 0 \ \mid \ \xi - \lambda
Z_x\ \text{~is nef}\}.$$
\end{definition}

Alternately, $\varepsilon(E,x)$ can be defined as follows:
$$\varepsilon(E,x) \,=\, \inf \frac{\xi\cdot C}
{\text{mult}_x p_{\ast}C}\, ,$$
where the infimum is taken over all curves $C \,\subset\, \mathbb{P}(E)$ that intersect $p^{-1}(x)$, but are not contained in $p^{-1}(x)$. 
For more details on Seshadri constants of vector bundles, see \cite{Hac, FM}.

Now we work with the notation in Theorem \ref{thm-bsdh}.
Let $E$ be a $T$--equivariant vector bundle on $Z(w,\,\underline{i})$. 
Let $C \,\subset\, Z(w,\,\underline{i})$ be a
$T$--invariant curve. Since $C \cong \mathbb{P}^1$, the restriction of $E$ to $C$ is of the form
$$E\big\vert_C\, =\, \mathcal{O}_C(a_1) \oplus \ldots \oplus \mathcal{O}_C(a_n)$$ for some
integers $a_1,\,\cdots,\,a_n$.

The next result uses this decomposition to describe the Seshadri constants of $E$ at $T$--fixed points.

\begin{theorem}\label{sc-bsdh}
Let $G$ be a semisimple and simply connected algebraic group. Let $T$ be a maximal torus of $G$, 
and let $B$ be a Borel subgroup of $G$ containing $T$ as in Theorem \ref{thm-bsdh}.
Let $w\,=\,s_{{i_1}}s_{{i_2}}\cdots s_{{i_r}}$ be a reduced expression, and 
$\underline i\,:=\,(i_1,\,\cdots,\,i_r)$. Let
$X\,:=\, Z(w,\,{\underline i})$ be the corresponding Bott-Samelson-Demazure-Hansen variety. 
Let $E$ be a
$T$--equivariant nef vector bundle on $X$ of rank $n$, and let $x \,\in\, X$ be a $T$--fixed
point. Then 
$$\varepsilon(E,x) \,=\, \text{min}\{a_i(C)\}_{C,i}\, ,$$ where the minimum is taken over
all $T$--invariant curves $C \,\subset\, X$ passing through $x$ and integers
$\{a_1(C),\,\cdots,\, a_n(C)\}$ such that $E\big\vert_C \,= \,\mathcal{O}_C(a_1(C)) \oplus \cdots \oplus
\mathcal{O}_C(a_n(C))$. 
\end{theorem}

\begin{proof}
Let $W_x$ denote the exceptional divisor of the blow-up $$\pi\,:\, \widetilde{X} \,\longrightarrow\,
X \,=\, Z(w,\,\underline{i})$$ at the point $x \,\in\, Z(w,\,\underline{i})$. By definition, the Seshadri constant of $E$ at
$x$ is given by the following (see \eqref{fig1}): 
$$\varepsilon(E,x) \,=\, \text{sup}\{\lambda \ge 0 ~\mid ~ \xi - \lambda q^{\star}(W_x)\
\text{~is nef}\}.$$

To prove that $\xi - \lambda q^{\star}(W_x)$ is nef for a particular $\lambda$, we need
to show that $$(\xi - \lambda q^{\star}(W_x)) \cdot D \,\ge\, 0,$$ for every
curve $D \,\subset\, \mathbb{P}(\pi^{\star}E)$. As
argued in the proof of Theorem \ref{thm-bsdh}, there exists a $T$--invariant
curve $\widetilde{D} \,\subset\, \mathbb{P}(\pi^{\star}E)$ which is linearly
equivalent to $D$ and moreover
$$\widetilde{C} \,:=\, q(\widetilde{D}) \,\subset\, \widetilde{X}$$ is a $T$--invariant curve. 
Note that $\widetilde{X}$ inherits an action of $T$, by Proposition \ref{blowup-bsdh}.

Consider the following diagram: 
\begin{equation}\label{fig2}
\xymatrix{
 \mathbb{P}(\pi^{\star}E\big\vert_{\widetilde{C}}) \ar@{^{(}->}[r]\ar[d]_{q_1} & \mathbb{P}(\pi^{\star}E)
\ar[r]^{\widetilde{\pi}}\ar[d]_{q} & \mathbb{P}(E)\ar[d]^{p}\\
 \widetilde{C} \ar@{^{(}->}[r] & \widetilde{X} \ar[r]^{\pi} &
 X
 }
\end{equation}
Then 
$$(\xi - \lambda q^{\star}(W_x)) \cdot D \,=\, (\xi - \lambda q^{\star}(W_x)) \cdot \widetilde{D}\,=\, 
\left[\mathcal{O}_{\mathbb{P}(\pi^{\star}E\big\vert_{\widetilde{C}})}(1)-\lambda
q_1^{\star}(W_x\big\vert_{\widetilde{C}})\right] \cdot \widetilde{D}.$$
So $\xi - \lambda q^{\star}(W_x)$ is nef 
if and only if $\mathcal{O}_{\mathbb{P}(\pi^{\star}E\big\vert_{\widetilde{C}})}(1)-\lambda
q_1^{\star}(W_x\big\vert_{\widetilde{C}})$ 
is nef for every $T$--invariant curve $\widetilde{C} \,\subset\,
\widetilde{X}$.

Now let $\widetilde{C} \,\subset\, \widetilde{X}$ be any $T$--invariant
curve. We know that $\widetilde{C}$ is isomorphic to the projective
line $\mathbb{P}^1$. We will now investigate the values $\lambda \,\ge\, 0$ for which
the line bundle $\mathcal{O}_{\mathbb{P}(\pi^{\star}E\big\vert_{\widetilde{C}})}(1)-\lambda
q^{\star}(W_x\big\vert_{\widetilde{C}})$ is nef. 

First suppose that $\widetilde{C}$ is contained in the exceptional
divisor $W_x$ of the blow-up $$\pi\,:\, \widetilde{X} \,\longrightarrow\, X.$$ 
Note that $W_x$ is isomorphic to a projective space, and 
$$\mathcal{O}_{\widetilde X}(W_x)\big\vert_{W_x} \,=\, \mathcal{O}_{W_x}(-1)\, .$$ Hence we have
$-\lambda(W_x\big\vert_{\widetilde{C}}) \,=\,\mathcal{O}_{\widetilde{C}}(\lambda)$. 
Since $E$ on $X$ is nef by hypothesis, it follows that 
$\mathcal{O}_{\mathbb{P}(\pi^{\star}E\big\vert_{\widetilde{C}})}(1)-\lambda
q_1^{\star}(W_x\big\vert_{\widetilde{C}})$ is nef for every $\lambda \,\geq\, 0$. 

Now suppose that $\widetilde{C}$ is not contained in $W_x$, and set $C
\,=\, \pi(\widetilde{C})$. Then $C \,\subset\, X$ is a $T$--invariant curve. If
$x\,\notin\, C$, then $W_x\big\vert_{\widetilde{C}} \,=\,
\mathcal{O}_{\widetilde{C}}$, and 
$\mathcal{O}_{\mathbb{P}(\pi^{\star}E\big\vert_{\widetilde{C}})}(1)$ is nef
because $\pi^{\star}E\big\vert_{\widetilde{C}}$ is so.

Finally, assume that $x \,\in\, C$. Then $$W_x \cdot \widetilde{C} \,=\, 1,$$ since $C$ is a
smooth curve. Let $a_1(C),\,\cdots,\,a_n(C)$ be non-negative integers such that
$E\big\vert_C \,=\, \mathcal{O}_C(a_1(C)) \oplus \ldots \oplus \mathcal{O}_C(a_n(C))$. Then 
$\mathcal{O}_{\mathbb{P}(\pi^{\star}E\big\vert_{\widetilde{C}})}(1)-\lambda
q_1^{\star}(W_x\big\vert_{\widetilde{C}})$ is nef if and only if 
$\mathcal{O}_{\widetilde{C}}(a_1(C)-\lambda) \oplus \cdots \oplus \mathcal{O}_{\widetilde{C}}(a_n(C)-\lambda)$ is nef. 
Now. $$\mathcal{O}_{\widetilde{C}}(a_1(C)-\lambda) \oplus \cdots \oplus
\mathcal{O}_{\widetilde{C}}(a_n(C)-\lambda)$$ is nef
if and only if $\lambda \,\le\,
\text{min}\{a_1(C),\,\cdots,\,a_n(C)\}$. Running over all $T$--invariant curves in
$X$, we obtain the theorem. 
\end{proof}

\begin{remark}\rm
Similar computations of Seshadri constants were carried out in \cite[Proposition 3.2]{HMP}
for torus-equivariant vector bundles on toric varieties, and in \cite[Theorem 3.2]{BHN}
for torus-equivariant vector bundles on generalized flag varieties.
\end{remark}

\begin{remark}\label{rem-bsdh}\rm
It is known that a nef vector bundle $E$ on a projective variety $X$ is ample if and only if $\inf_{x\in X} \varepsilon(E,x)
\, >\, 0$; for example, see
\cite[Theorem 3.11]{FM}. On the other hand, in general, Seshadri constants of ample line bundles can be arbitrarily small positive numbers 
(see \cite[Example 5.2.1]{La1}). 

But one expects better lower bounds at very general points. If $L$ is an ample line bundle on a surface $X$, then 
\cite{EL} proved the following: 
$$\varepsilon(L,x) \ge 1, \,\text{~for a very general point~} x \in X.$$ 

The same inequality is conjectured for ample line bundles on higher dimensional varieties; see \cite[Conjecture 5.2.4]{La1}.

It is interesting to ask if this lower bound can be generalized to vector bundles of arbitrary rank. This question is open in general. 

Theorem \ref{sc-bsdh} shows that for an ample $T$--equivariant
vector bundle $E$ on $Z(w,\,\underline{i})$, we have 
$$\varepsilon(E,x) \,\ge\, 1$$ for every $T$--fixed point $x \,\in\, Z(w,\,\underline{i})$. 
If $E$ is a vector bundle on a projective variety $X$, then the maximum value of $\varepsilon(E,x)$, 
as the point $x$ runs over $X$, is achieved at a very general point in $X$
(see \cite[Proposition 3.35]{FM}). Since $T$--fixed points are special, we conclude that
$$\varepsilon(E,x) \,\ge \,1$$
for a very general point $x \,\in\, Z(w,\,\underline{i})$ if $E$ is an ample $T$--equivariant vector bundle on $Z(w,\,\underline{i})$.
This gives an affirmative answer to the above question in the case of 
ample $T$--equivariant vector bundles on BSDH varieties. 
\end{remark}

\subsection{Equivariant bundles on wonderful compactifications}\label{wonderful}

In this subsection we study positivity of equivariant vector bundles on wonderful compactifications. The notation of Section 
\ref{prelims-wonderful} is followed. The proofs are similar to the corresponding proofs in Section \ref{bsdh}.

\begin{theorem}\label{thm-wonderful}
Let $G$ be a semisimple adjoint type algebraic group. 
Let $\sigma$ be an automorphism of $G$ of order two, and let $H$ be the subgroup of all fixed points of $\sigma$ in
$G.$ Assume that $G/H$ has minimal rank. Let $X\,:=\,\overline{G/H}$ be the wonderful compactification of $G/H$.
Let $T$ be a $\sigma$--invariant maximal torus of $G$ such that
$\dim T_1$ (see \eqref{t12}) is maximal.

Let $E$ be a $T$--equivariant vector bundle on $X$. Then $E$ is nef
(respectively, ample) if and only if the restriction $E\big\vert_C$ of $E$
to every $T$--invariant curve $C$ on $X$ is nef (respectively, ample). 
\end{theorem}

\begin{proof}
The proof for the nef part is the same as the proof of the analogous statement in Theorem \ref{thm-bsdh}: 
Suppose that $E$ is a $T$--equivariant vector bundle on $X$ such that the restriction $E\big\vert_C$ of $E$
to every $T$--invariant curve $C$ is nef. If we are given a curve $D$, then we can 
construct a $T$--invariant curve $C$ such that $C$ and $D$ are linearly equivalent. So it follows that
$E$ is nef. 

Next suppose that $E\big\vert_C$ is ample for every $T$--invariant curve $C \,\subset\, X$. We will
show that $E$ itself is ample. This also is similar to the proof in Theorem \ref{thm-bsdh}; in fact, it is easier.

Note that every line bundle on $X$ is $T$--equivariant. 
Hence if $F$ is a $T$--equivariant vector bundle on $X$, then so is
$F \otimes L$ for any line bundle $L$ on $X$. 
To show that $E$ is ample, fix an
ample line bundle $L$ on $X$. 
Since every line bundle on $X$ is $T$--equivariant, $E \otimes L$ is also $T$--equivariant. 

By Lemma \ref{lem-wonderful}, there are only finitely many $T$--invariant curves in $X$. Now we complete
the argument exactly as done in the last paragraph of the proof of Theorem \ref{thm-bsdh}. 
\end{proof}

\begin{proposition}\label{blowup-wonderful}
Let $G$, $X$, $T$ and $B$ be as in Theorem \ref{thm-wonderful}.
Let $x \,\in\, X$ be a $T$--fixed point, and let $\pi\,:\, \widetilde{X}\,\longrightarrow\,
X$ denote the blow-up of $X$ at $x$. Then the following three statements hold:
\begin{enumerate}
\item The action of $T$ on $X$ lifts to $\widetilde{X}$. 

\item Let $F$ be a $T$--equivariant vector bundle on $\widetilde{X}$. Then $F$ is nef if and 
only if the restriction $F\big\vert_{\widetilde{C}}$ of $F$ to every $T$--invariant curve 
$\widetilde{C}\,\subset\, \widetilde{X}$ is nef.

\item Let $W_x$ denote the exceptional divisor of the blow-up $\pi$. Let $E$ be a 
$T$--equivariant vector bundle on $X$. Then $(\pi^{\star}E)\otimes 
\mathcal{O}_{\widetilde{X}}(m\cdot W_x)$ is a $T$--equivariant vector bundle on $\widetilde{X}$ for 
every integer $m$.
\end{enumerate}
\end{proposition}

\begin{proof}
Note that the center of the blow up $\pi$ is a $T$--fixed point, so the $T$--action lifts to the blow up $\widetilde{X}$.
The rest of the proof is identical to the proof of Proposition \ref{blowup-bsdh}, so we do not give details here.
\end{proof}

Now we work with the notation in Theorem \ref{thm-wonderful}.
Let $E$ be a $T$--equivariant vector bundle on $X$. 

Let $C \,\subset\, X$ be a
$T$--invariant curve. By Lemma \ref{lem-wonderful}, $C$ is isomorphic to the projective line $\mathbb{P}^1$. Hence 
the restriction of $E$ to $C$ is of the form
$$E\big\vert_C\, =\, \mathcal{O}_C(a_1) \oplus \ldots \oplus \mathcal{O}_C(a_n)$$ for some
integers $a_1,\,\cdots,\,a_n$.

\begin{theorem}\label{sc-wonderful}
Let $G$ be a semisimple adjoint type algebraic group. 
Let $\sigma$ be an automorphism of $G$ of order two, and let $H$ be the subgroup of all fixed points of $\sigma$ in $G.$
Assume that $G/H$ has minimal rank. 
Let $X\,:=\,\overline{G/H}$ be the wonderful compactification of $G/H$.
Let $T$ be a $\sigma$--invariant maximal torus of $G$ such that
$\dim T_1$ is maximal (see \eqref{t12}).

Let $E$ be a
$T$--equivariant nef vector bundle on $X$ of rank $n$, and let $x \,\in\, X$ be a $T$--fixed
point. Then 
$$\varepsilon(E,x)
\,=\, \text{min}\{a_i(C)\}_{C,i}\, ,$$ where the minimum is taken over
all $T$--invariant curves $C \,\subset\, X$ passing through $x$ and integers
$\{a_1(C),\,\cdots,\, a_n(C)\}$ such that $E\big\vert_C \,= \,\mathcal{O}_C(a_1(C)) \oplus \cdots \oplus
\mathcal{O}_C(a_n(C))$. 
\end{theorem}

\begin{proof}
We argue as in the proof of Theorem \ref{sc-bsdh}. 

Let $W_x$ denote the exceptional divisor of the blow-up $$\pi\,:\, \widetilde{X} \,\longrightarrow\,
X $$ at the point $x \,\in\, X$. By definition, the Seshadri constant of $E$ at
$x$ is the following (see \eqref{fig1}): 
$$\varepsilon(E,x) \,=\, \text{sup}\{\lambda \,\ge\, 0 \ \mid\ \xi - \lambda q^{\star}(W_x)\
\text{~is nef}\}.$$

To prove that $\xi - \lambda q^{\star}(W_x)$ is nef for a given $\lambda$, we need
to show that $$(\xi - \lambda q^{\star}(W_x)) \cdot D \,\ge\, 0,$$ for every
curve $D \,\subset\, \mathbb{P}(\pi^{\star}E)$. As before, there exists a $T$--invariant
curve $\widetilde{D} \,\subset\, \mathbb{P}(\pi^{\star}E)$ which is linearly
equivalent to $D$ and satisfies the condition that 
$\widetilde{C} \,:=\, q(\widetilde{D}) \,\subset\, \widetilde{X}$ is a $T$--invariant curve. 

As in the proof of Theorem \ref{sc-bsdh}, $\xi - \lambda q^{\star}(W_x)$ is nef 
if and only if $\mathcal{O}_{\mathbb{P}(\pi^{\star}E\big\vert_{\widetilde{C}})}(1)-\lambda
q_1^{\star}(W_x\big\vert_{\widetilde{C}})$ 
is nef for every $T$--invariant curve $\widetilde{C} \,\subset\,
\widetilde{X}$; see the diagram in \eqref{fig2}. 

We can now complete the proof exactly as done in Theorem \ref{sc-bsdh}. 
\end{proof}

\begin{remark}\label{rem-wonderful}\rm
Let $X\,=\,\overline{G/H}$ be as in Theorem \ref{sc-wonderful}, and let $E$ be an ample $T$--equivariant vector bundle on $X$. 
Arguing as in Remark \ref{rem-bsdh}, we have 
$$\varepsilon(E,x) \,\ge\, 1,$$ for a very general point $x \,\in\, X$. 
\end{remark}

\section*{Acknowledgements}
We are very grateful to the referee for a careful reading and many useful suggestions which improved the paper. 
The second and the third authors are partially supported by a grant from Infosys Foundation.
The first author is partially supported by a J. C. Bose Fellowship.

\section*{Conflict of interest statement}
On behalf of all authors, the corresponding author states that there is no conflict of interest.

\end{document}